\documentclass[reqno]{amsart}
\usepackage[english]{babel}
\usepackage{amscd,amssymb,amsmath,amsfonts,latexsym,amsthm}
\usepackage{inputenc}
\usepackage{graphicx}

\newtheorem{thm}{Theorem}[section]
\newtheorem{cor}[thm]{Corollary}
\newtheorem{lem}[thm]{Lemma}
\newtheorem{prop}[thm]{Proposition}
\newtheorem{defn}[thm]{Definition}
\newtheorem{exam}[thm]{Example}
\newtheorem{exams}[thm]{Examples}

\numberwithin{equation}{section}

\begin{document}

\title[Pro-solvable Lie algebras]{On derivations and low-dimensional (co)homology groups of pro-sovable Lie algebras associated with $\mathbf{n}_1$ and $\mathbf{n}_2$}

\author{K.K. Abdurasulov$^1$, I.S. Rakhimov$^{1,2}$, G.O. Solijanova$^3$, }

\address{$^1$ V.I.Romanovski Institute of Mathematics Uzbekistan Academy of Sciences, Tashkent, Uzbekistan.} \email{abdurasulov0505@mail.ru}

\address{$^2$ Faculty of Computer and Mathematical Sciences, Universiti Teknologi MARA (UiTM), Shah Alam, Malaysia}
\email{isamiddin@uitm.edu.my}

\address{$^3$ National University of Uzbekistan, Tashkent, Uzbekistan} \email{gulhayo.solijonova@mail.ru}

\begin{abstract} In the paper we describe the derivations of two $\mathbb{N}$-graded infinity-dimensional Lie algebras $\mathbf{n}_1$ and $\mathbf{n}_1$ what are positive parts of affine Kats-Moody algebras $A^{(1)}_1$ and $A^{(2)}_2$, respectively. Then we construct all pro-solvable Lie algebras whose potential nilpotent ideals are $\mathbf{n}_1$ and $\mathbf{n}_2$. For two specific representatives of these classes low-dimensional (co)homology groups are computed.
\end{abstract}

\subjclass[2010]{17B40, 17B56, 17B65}

\keywords{Lie algebra, potentially nilpotent Lie algebra, pro-nilpotent (pro-solvable) Lie algebra, cohomology group.}

\maketitle
\normalsize

\section{Introduction}

The paper contains a pioneer results towards a generalization to infinity-dimensional case of the method applied in finite-dimensional case to construct solvable Lie algebras by using nilpotent algebras (see \cite{Mub}).
In this section we introduce necessary definitions to use them later. Most of the definitions are borrowed from \cite{Millio}. All the algebras considered in the paper are assumed to be over the field of complex numbers.
\begin{defn} A Lie algebra $( L,[\cdot,\cdot])$  is called a
positively graded ($\mathbb{N}$-graded) if it is represented as a direct sum $$L=\bigoplus\limits_{i=1}^\infty L_i$$ its homogonious subspaces $L_i$ such that $[L_i,L_j]\subset L_{i+j}$ for all $i,j \in \mathbb{N}.$
\end{defn}
\begin{exam} \emph{} \label{exam1}
 Let $L = L(m)$ be the free Lie algebra with the generators $a_1,...,a_m.$ Consider the following linear span of the $k$-words $L_k=span \{[a_{i_1},[a_{i_2},[...,...]],a_{i_k}]\}$. Then it is easy to see that $L_k$, where $k=1, 2,...$ are desired homogeneous subspaces of $L=L(m).$
\end{exam}
\begin{exam}
Let $L=span \{e_1, e_2, e_3, e_4,...\}$ with the commutation rules
$$[e_1,e_i] = e_{i+1}, i \geq 2, \ \ [e_i,e_k] = 0, \ \mbox{where}\ i,k \neq 1.$$ As a homogeneous subspaces we take $L_k=span \{e_k\},$ $k=1,2,....$
\end{exam}
The quotient algebra $L/I$ of an infinity-dimensional positively graded Lie algebra $L=\bigoplus\limits_{i=1}^\infty L_i$ by the ideal $I=\bigoplus\limits_{i=k+1}^\infty L_i$ is a nilpotent Lie algebra.

The concept of Lie algebra's width has been introduced by Zelmanov and Shalev in \cite{ZSh1,ZSh2}. A
positively graded Lie algebra $L=\bigoplus\limits_{i=1}^\infty L_i$
is said to be of width $d$ if there
is a positive number $d$ such that $\dim L_i \leq d$, for all $i \in \mathbb{N}.$ The problem
of classifying graded Lie algebras of finite width was outlined by Zelmanov and
Shalev as an important and difficult problem (even ``a formidable challenge'' \cite{ZSh2}).

The class of positive graded finite-dimensional Lie algebras has been classified by D.Millionshchikov (see \cite{Millio,DisM}).  As for infinity-dimensional case D.Millionshchikov has managed to classify those with the width $\frac{3}{2}$ (see \cite{Millio}).

\begin{defn} A grading $L=\bigoplus\limits_{i=1}^\infty L_i$ is called natural if $[L_1,L_i]\subset L_{i+1}$ for all $i \in \mathbb{N}.$
\end{defn}

\begin{exams}
Let us consider the algebra $L=span \{e_1,e_2,e_3,e_4,...\}$ from Example \ref{exam1}.
$$L_1= span \{e_1,e_2\}, \ \ L_2= span \{e_3\}, \ \ L_3= span \{e_4\}, \ \ L_4= span \{e_5\}, \dots $$ is the natural grading of $L$.
\end{exams}


For a Lie algebra $L$ we define the {\it lower central} and the {\it derived series} as follows
$$L^1=L, \ L^{k+1}=[L^k,L],  \ k \geq 1, \qquad L^{[1]}=L, \ L^{[s+1]}=[L^{[s]},L^{[s]}], \ s \geq 1,$$
respectively.

\begin{defn}  A Lie algebra $L$ is said to be potentially solvable (respectively, potentially nilpotent) if
$\bigcap_{i=1} ^{\infty}L^{[i]}=0$ (respectively, $\bigcap_{i=1} ^{\infty}L^{i}=0$).
\end{defn}
\begin{exams}
\begin{enumerate}\emph{}

\item[1.] Here are two examples of potentially nilpotent Lie algebras (see \cite{Millio}).
\begin{itemize}
\item Lie algebra $\mathbf{m}_2=span\{e_1, e_2, e_3, \dots \}$ with the commutation rules
$$\mathbf{m}_2:=\left\{\begin{array}{ll}
[e_{1} ,e_{i}]=e_{i+1},\quad i\geq2,\\[1mm]
[e_{2} ,e_{j}]=e_{j+2},\quad i\geq3.\\[1mm]
 \end{array}\right.$$
 is a potentially nilpotent Lie algebra.
 \item The positive part $W^{+}$ of the Witt algebra is given by the commutation rules $$[e_i,e_j]=(i-j)e_{i+j}, \quad i,j\in \mathbb{N}.$$
 The algebra $W^{+}$ also is a potentially nilpotent Lie algebra.
\end{itemize}

\item[2.] Another two algebras below give examples of potentially solvable but not necessarily potentially nilpotent algebras (see \cite{Qobil2, Ayupov}).
\begin{itemize}
\item Let $L=span \{x, y, e_1, e_2, e_3, \dots \}$ with the followinf rules of compositions
$$L:=\left\{\begin{array}{ll}
[e_{i} ,e_{1}]=e_{i+1},\quad i\geq2,\\[1mm]
[x,e_1]=e_1,\\[1mm]
[x,e_i]=(i-1)e_i,\quad i\geq 2, \\[1mm]
[y,e_i]=e_i,\quad i\geq 2.
 \end{array}\right.$$

\item Consider non-negative part $\widetilde{W^{+}}$ of the Witt algebra $W$ given by the following rules
$$[e_i,e_j]=(i-j)e_{i+j}, \quad i,j\geq 0.$$
The algebra $\widetilde{W^{+}}$ is potentially solvable but is not potentially nilpotent.
\end{itemize}
\end{enumerate}

\end{exams}

\begin{defn}  A Lie algebra $L$ is said to be pro-solvable (respectively, pro-nilpotent) if
$\bigcap_{i=1} ^{\infty}L^{[i]}=0$ and $\dim L/L^{[i]}<\infty$ (respectively,
$\bigcap_{i=1} ^{\infty}L^{i}=0$ and $\dim L/L^{i}<\infty$) for any $i\geq1$.
\end{defn}

Let $L$ be a pro-nilpotent Lie algebra. Then for the ideals $L^k$, where $k \geq 1$ one has
$$L=L^1 \supseteq L^2 \supseteq ...\supseteq L^k \supseteq L^{k+1}\supseteq ..., \ \mbox{and}\ [L^i,L^i]\subset L^{i+j}, \ i,j \in \mathbb{N}.$$
Consider the associated graded Lie algebra $gr L= \bigoplus\limits_{i=1}^{+\infty}(L^i/L^{i+1})$ with respect to the above filtration. The Lie bracket on $gr L$ is given by
$$[x+L^{i+1},y+L^{i+1}]=[x,y]+L^{i+j+1}, \ \mbox{for}\ x \in L^{i}, \ y \in L^{j}.$$
It is evident that $gr L$ is a $\mathbb{N}$-graded Lie algebra.
\begin{defn} A pro-nilpotent Lie algebra $L$ is said to be a
 \emph{naturally graded} if it is isomorphic to $gr L$. The grading $L=\bigoplus\limits_{i=1}^\infty L_i$ of a naturally graded Lie algebra $L$ is called natural if there exists a graded isomorphism
 $$\varphi:gr L \longrightarrow L, \ \ \varphi((gr L)_i)=L_i, \ \ i \in \mathbb{N}.$$
\end{defn}

Let us define \emph{natural grading} of the algebra $L$. Write
$$L_1:=L^1/L^2,\,L_i:=L^i/L^{i+1},\ \mbox{where}\ i=2,3,\dots$$

In this case, $L\simeq L_1\oplus L_2\oplus L_3\oplus\dots$
 One can readily verify the validity of the inclusions $[L_i,L_j]\subset L_{i+j}$ and the grading is natural.

Let $L$ be a pro-nilpotent Lie algebra. Suppose that there is a basis
$\{e_1,e_2,e_3,\dots\}$ of $L$,  $k_i=\dim L_i$ and $L^i=span \{e_{k_i},e_{k_{i+1}},e_{k_{i+2}},\dots\}$,  $i=1, 2, 3, ...$ Define
$$c(L)=(k_2-k_1,k_3-k_2,k_4-k_3,\dots),\, where\, k_1 = 1.$$
Since $L$ is potential nilpotent, it follows that
$k_{i+1}-k_i\geq1$ for any $i\geq 1$.

Here are two important examples of the naturally graded Lie algebras from \cite{Millio}.

The first of them is a Lie algebra $\mathbf{n}_1$ with the following non-zero brackets on a basis $\{e_1,e_2,e_3,...\}$
$$\mathbf{n}_1:\quad \begin{array}{lllll}
[e_i, e_j] =c_{i,j} e_{i+j},& i,j\in \mathbb{N}, & c_{i,j}=\left \{\begin{array}{ccc}
                                                    1, & if & i-j\equiv 1 \ mod \ 3,\\[1mm]
                                                    0,&if & i-j\equiv 0 \ mod\ 3,\\[1mm]
                                                    -1,& if& i-j\equiv -1 \ mod\ 3.
                                                     \end{array}\right.
\end{array}$$

One can see that
$$(\mathbf{n}_1)_1:=\mathbf{n}_1^1/\mathbf{n}_1^2=span \{e_1,e_2\},\quad (\mathbf{n}_1)_i:=\mathbf{n}_1^i/\mathbf{n}_1^{i+1}=span \{e_{i+1}\},\,i=2,3,\dots\,.$$
Hence, $$c(\mathbf{n}_1)=\{2,1,1,\dots\}$$

The second algebra is $$\mathbf{n}_2:\quad \quad \begin{array}{lllll}
[f_q,f_l] =d_{q,l} f_{q+l},& q,l\in \mathbb{N},
\end{array}$$ given by the following table of multiplications on a basis $\{f_{8i+1},f_{8i+2} ,f_{8i+3}, f_{8i+4}, f_{8i+5}, f_{8i+6}, f_{8i+7}, f_{8i+8}\}, \ \ i \in \mathbb{N}\cup \{0\}:$
$$\begin{array}{lllllllll}
& f_{8j}&f_{8j+1} &f_{8j+2} &f_{8j+3}& f_{8j+4} &f_{8j+5} &f_{8j+6} &f_{8j+7}\\[1mm]
f_{8i}  &  \quad 0& \quad1& -2& -1& \quad0&\quad1&\quad2&  -1\\[1mm]
f_{8i+1}&  -1&\quad0&\quad1 &\quad1 & -3&  -2& \quad0&\quad1\\[1mm]
f_{8i+2}&   \quad2& -1& \quad0& \quad0& \quad0&\quad1 & -1& \quad0\\[1mm]
f_{8i+3} &\quad1&  -1& \quad0& \quad0& \quad3&  -1& \quad1&  -2\\[1mm]
f_{8i+4} &\quad0 &\quad3 &\quad0&  -3& \quad0& \quad3& \quad0&  -3\\[1mm]
f_{8i+5} & -1 &\quad2 & -1 &\quad1 & -3& \quad0& \quad0&  \quad 1\\[1mm]
f_{8i+6}&  -2 &\quad0&\quad1&  -1&\quad 0& \quad0 &\quad0&\quad1\\[1mm]
f_{8i+7}& \quad1 & -1 &\quad0 &\quad2& \quad3&-1 & -1&\quad 0\\[1mm]
\end{array}$$
where $i,j$ are non-negative integers.

It is clearly seen that
$$\mathbf{n}_2^2=span \{f_3,f_4,f_5,\dots\},\quad \mathbf{n}_2^3=span \{f_4,f_5,f_6,\dots\},\dots,
\mathbf{n}_2^{i}=span \{f_{i+1},f_{i+2},\dots\},\dots$$
Hence, $$(\mathbf{n}_2)_1=span \{f_1,f_2\},\, (\mathbf{n}_2)_i=span \{f_{i+1}\}$$
$$c(\mathbf{n}_2)=(2,1,1,\dots)$$
The algebras  $\mathbf{n}_1$ and $\mathbf{n}_2$ are naturally graded pro-nilpotent Lie algebras with homogeneous components $(\mathbf{n}_1)_j$ and $(\mathbf{n}_2)_j$, where $j=1, 2, 3, \dots,$ respectively.
 These two algebras are known as positive parts of the affine Kats-Moody algebras $A^{(1)}_1$ and $A^{(2)}_2,$ respectively. They also are famous with their
 special role in the theory of combinatoric identities (see \cite{Garland, Lepowsky1, Lepowsky2}).
The algebras $\mathbf{n}_1$ and $\mathbf{n}_2$ are two representatives (among a few others) of the isomorphism classes of natural graded algebras over fields of characteristic zero given by A. Fialowski and D.Millionshchikov (see \cite{F, Millio}). Note that the similar problems that in this paper for another two representative denoted by the authors of \cite{F, Millio} as $\mathbf{m}_0$ and $\mathbf{m}_2$ have been treated earlier in \cite{Qobil2}.
\begin{defn} A linear map $d \colon L \rightarrow L$ of $(L, [\cdot,\cdot])$ is said to be a
 \textbf{derivation} of $L$ if for all $x, y \in L$, the following  condition \[d([x,y])=[d(x),y] + [x, d(y)] \,\] holds.
\end{defn}
Recall that for $x \in L$ the function $ad_x: L \rightarrow L$ defined by $ad_x(y)=[x,y], \ \forall y \in L$
is a derivation called  {\it inner derivation}.

\begin{defn} A Lie algebra $L$ is called  \textbf{complete} if ${\rm Center}(L)=0$ and all the derivations of $L$ are
inner.
\end{defn}

\begin{defn} A linear map $\rho: L \to L$ is called potentially nilpotent, if $\cap_{i=1}^{\infty}(Im\ \rho^i)=0$ holds true.
\end{defn}

Below we introduce the analogue of the notion of nil-independency which had played a crucial role in the description of finite-dimensional solvable Lie algebras in \cite{Mub}.

\begin{defn} Derivations $d_{1},d_{2},\dots,d_{n}$ of a Lie algebra $L$ over a field $\mathbb{F}$ are said to be potentially nil-independent, if a map $f=\alpha_{1}d_{1}+\alpha_{2}d_{2}+\ldots+\alpha_{n}d_{n}$ is not potentially nilpotent for any scalars $\alpha_{1},\alpha_{2},\dots,\alpha_{n}\in \mathbb{F}$. In other words,
$\cap_{i=1}^{\infty}Imf^i=0$ if and only if $\alpha_{1}=\alpha_{2}=\dots=\alpha_{n}=0.$
\end{defn}


Recall that
$${\rm H}^1(L,L)={\rm Der}(L)/{\rm Inner}(L) \quad \mbox{and} \quad {\rm H}^2(L,L)={\rm Z}^2(L,L)/{\rm B}^2(L,L)$$ where the set ${\rm Z}^2(L,L)$ consists of those elements $\varphi\in {\rm Hom}(\wedge^2L, L)$ such that
\begin{equation}\label{eq5}
[x,\varphi(y,z)] - [\varphi(x,y), z]
+[\varphi(x,z), y] +\varphi(x,[y,z]) - \varphi([x,y],z)+\varphi([x,z],y)=0,
\end{equation}
while ${\rm B}^2(L,L)$ consists of elements $\psi\in {\rm Hom}(\wedge^2L, L)$ such that
\begin{equation}\label{eq4}
\psi(x,y)=[d(x),y] + [x,d(y)] - d([x,y]) \,\, \mbox{for some linear map} \,\, d\in {\rm Hom}(L,L).
\end{equation}

In terms of cohomology groups the notion of completeness of a Lie algebra $L$ means that it is centerless and ${\rm H}^1(L,L)=0$.

For the convenience we introduce the notation
\begin{equation}\label{Coc}Z(a,b,c)=[a,\varphi(b,c)]-[\varphi(a,b),c]+[\varphi(a,c),b]
+\varphi(a,[b,c])-\varphi([a,b],c)+\varphi([a,c],b).\end{equation}

And at last, the organization of the paper is as follows. The next section contains the description of the derivations of two $\mathbb{N}$-graded
infinity-dimensional Lie algebras $\mathbf{n}_1$ and $\mathbf{n}_2$ what are positive parts of affine Kats-Moody algebras $A^{(1)}_1$ and $A^{(2)}_2$,
respectively. This followed by the construction of all pro-solvable Lie algebras whose potential nilpotent ideals are $\mathbf{n}_1$ and $\mathbf{n}_2$.
The final sections contains the results on completeness of two specific representatives of these classes and the triviality of their second
(co)homology groups.

\section{Main results}

\subsection{Derivations of pro-nilpotent Lie algebras $\mathbf{n}_1$ and $\mathbf{n}_2$}\

 In this section we describe the derivations of $\mathbf{n}_1$ and $\mathbf{n}_2$. For the algebra $\mathbf{n}_1$ the following result holds true.

\begin{prop}\label{prop1}
The derivations of the algerba $\mathbf{n}_1$ are given as follows:

$\begin{array}{lll}
d(e_{3i-2})&=&\sum\limits_{k=1}^{t}(((i-1)\beta_{3k-1}+i\alpha_{3k-2})e_{3k+3i-5}+\alpha_{3k}e_{3k+3i-3}),\\
d(e_{3i-1})&=&\sum\limits_{k=1}^{t}(i\beta_{3k-1}+ (i-1)\alpha_{3k-2})e_{3k+3i-4}+\beta_{3k}e_{3k+3i-3}),\\
d(e_{3i})&=&\sum\limits_{k=1}^{t}(i(\beta_{3k-1}+ \alpha_{3k-2})e_{3k+3i-3}-\beta_{3k}e_{3k+3i-2}-\alpha_{3k}e_{3k+3i-1}),\,\, i\in \mathbb{N}.
\end{array}$
\end{prop}
\begin{proof} Since the algebra $\mathbf{n}_1$ has two generators $\{e_{1}, e_{2}\}$, any derivation $d$ of $\mathbf{n}_1$  is completely determined by $d(e_{1}) $ and
$d(e_{2})$.

Let
$$d(e_{1})=\sum_{i=1}^{p}\alpha_{1_{i}}e_{{i}}, \quad d(e_{2})=\sum_{j=1}^{q}\beta_{{j}}e_{j}.$$

Without loss of generality one can assume that
$$d(e_{1})=\sum_{i=1}^{3t}\alpha_{i}e_{i},\quad d(e_{2})=\sum_{j=1}^{3t}\beta_{j}e_{j}, \quad  max\{p,q\}\le 3t, \quad t\in \mathbb{N}.$$

Applying the definition of the derivation we have

$\begin{array}{lllll} d(e_3)&=&d([e_2,e_1])&=&
\sum\limits_{k=1}^{t}(\beta_{3k-1}+\alpha_{3k-2})e_{3k}-\sum\limits_{k=1}^{t}\beta_{3k}e_{3k+1}
-\sum\limits_{k=1}^{t}\alpha_{3k}e_{3k+2},\\
d(e_4)&=&-d([e_3,e_1])
&=&\sum\limits_{k=1}^{t}(\beta_{3k-1}+2\alpha_{3k-2})e_{3k+1}+\sum\limits_{k=1}^{t}\alpha_{3k-1}e_{3k+2}
-\sum\limits_{k=1}^{t}\alpha_{3k}e_{3k+3}/\end{array}$\\ and
$0=d([e_4,e_1])
=-2\sum\limits_{k=1}^{t}\alpha_{3k-1}e_{3k+3}=0$ implies $\alpha_{3k-1}=0,\,\, 1\leq k\leq t.$

Moreover,

$\begin{array}{lll}
d(e_5)&=&d([e_3,e_2])=[d(e_3),e_2]+[e_3,d(e_2)]\\
&=&\sum\limits_{k=1}^{t}(\beta_{3k-1}+\alpha_{3k-2})e_{3k+2}+\sum\limits_{k=1}^{t}\beta_{3k}e_{3k+3}
-\sum\limits_{k=1}^{t}\beta_{3k-2}e_{3k+1}+\sum\limits_{k=1}^{t}\beta_{3k-1}e_{3k+2}\\
&=&\sum\limits_{k=1}^{t}(2\beta_{3k-1}+\alpha_{3k-2})e_{3k+2}+\sum\limits_{k=1}^{t}\beta_{3k}e_{3k+3}
-\sum\limits_{k=1}^{t}\beta_{3k-2}e_{3k+1},\end{array}$\\

and
$0=d([e_5,e_2])=\sum_{k=1}^{t}\beta_{3k-2}e_{3k+3}=0,$ gives  $\beta_{3k-2}=0,\,\, 1\le k\le t.$

In general, by applying induction we prove that

\begin{equation} \label{n1}
\begin{array}{lll}
d(e_{3i-2})&=&\sum\limits_{k=1}^{t}(((i-1)\beta_{3k-1}+i\alpha_{3k-2})e_{3k+3i-5}+\alpha_{3k}e_{3k+3i-3}),\\
d(e_{3i-1})&=&\sum\limits_{k=1}^{t}(i\beta_{3k-1}+ (i-1)\alpha_{3k-2})e_{3k+3i-4}+\beta_{3k}e_{3k+3i-3}),\\
d(e_{3i})&=&\sum\limits_{k=1}^{t}(i(\beta_{3k-1}+ \alpha_{3k-2})e_{3k+3i-3}-\beta_{3k}e_{3k+3i-2}-\alpha_{3k}e_{3k+3i-1}).\end{array}
\end{equation}

Indeed, if $i=1$ then the formulas above give the first step of the induction.

 Suppose that (\ref{n1}) is true for $i$. Based on this for $i+1$ the formula (\ref{n1}) is written as follows

$\begin{array}{lll} d(e_{3(i+1)-2})&=&d([e_{3i-2},e_{3}])=[d(e_{3i-2}),e_{3}]+[e_{3i-2},d(e_{3})]
\\
&=&\sum\limits_{k=1}^{t}(((i-1)\beta_{3k-1}+i\alpha_{3k-2})e_{3k+3i-2}+\alpha_{3k}e_{3k+3i}+(\beta_{3k-1}
+\alpha_{3k-2})e_{3k+3i-2})\\
&=&\sum\limits_{k=1}^{t}((i\beta_{3k-1}+(i+1)\alpha_{3k-2})e_{3k+3i-2}+\alpha_{3k}e_{3k+3i})\\
&=&\sum\limits_{k=1}^{t}(((i+1-1)\beta_{3k-1}+(i+1)\alpha_{3k-2})e_{3k+3(i+1)-5}+\alpha_{3k}e_{3k+3(i+1)-3}).\end{array}$

$\begin{array}{lll} d(e_{3(i+1)-1})&=&d([e_{3},e_{3i-1}])=[d(e_{3}),e_{3i-1}]+[e_{3},d(e_{3i-1})]\\
&=&\sum\limits_{k=1}^{t}((\beta_{3k-1}+\alpha_{3k-2})e_{3k+3i+1}+\beta_{3k}e_{3k+3i}+(i\beta_{3k+2}+ (i-1)\alpha_{3k-2})e_{3k+3i-1})\\
&=&\sum\limits_{k=1}^{t}(((i+1)\beta_{3k-1}+i\alpha_{3k-2})e_{3k+3i-1}+\beta_{3k}e_{3k+3i}),\end{array}$\\
and

\quad $\begin{array}{lll}d(e_{3(i+1)})&=&d([e_{3i+2},e_1])=[d(e_{3i+2}),e_1]+[e_{3i+2},d(e_1)]\\
&=&\sum\limits_{k=1}^{t}(((i+1)\beta_{3k-1}+i\alpha_{3k-2})e_{3k+3i}-\beta_{3k}e_{3k+3i+1}
+\alpha_{3k-2}e_{3k+3i}-\alpha_{3k}e_{3k+3i+2})\\
&=&\sum\limits_{k=1}^{t}((i+1)(\beta_{3k-1}+\alpha_{3k-2})e_{3k+3i}-\beta_{3k}e_{3k+3i+1}-\alpha_{3k}e_{3k+3i+2}).\end{array}$

\end{proof}

The proof of the following proposition is carried out similarly to that of Proposition \ref{prop1}.

\begin{prop}\label{prop3}
  The derivations of the algebra $\mathbf{n}_2$ have the following form
\quad $$\begin{array}{llllll}
d(f_{8i+1})=&\sum\limits_{k=0}^{t}(((4i+1)\alpha_{8i+8k+1}+2i\beta_{8k+2})f_{8i+8k+1}
+\alpha_{8k+3}f_{8i+8k+3}+\alpha_{8k+4}f_{8i+8k+4}\\[3mm]
            &\hfill +\alpha_{8k+5}f_{8i+8k+5}-2\beta_{8k+7}f_{8i+8k+6}+\alpha_{8i+8k+8}f_{8k+8}),\\[3mm]
d(f_{8i+2})=&\sum\limits_{k=0}^{t}((4i\alpha_{8k+1}+(2i+1)\beta_{8k+2})f_{8i+8k+2}
+\beta_{8k+3}f_{8i+8k+3}+\beta_{8k+7}f_{8i+8k+7}+\beta_{8k+8}f_{8i+8k+8}),\\[3mm]\end{array}$$
$\begin{array}{llllll}
d(f_{8i+3})=&\sum\limits_{k=0}^{t}(((4i+1)\alpha_{8k+1}+(2i+1)\beta_{8k+2})f_{8i+8k+3}
+\beta_{8k+3}f_{8i+8k+4}-\alpha_{8k+5}f_{8i+8k+7}\\[3mm]
            &\hfill -\beta_{8k+7}f_{8i+8k+8}-\beta_{8k+8}f_{8i+8k+9}-2\alpha_{8k+8}f_{8i+8k+10}),\\[3mm]
d(f_{8i+4})=&\sum\limits_{k=0}^{t}(((4i+2)\alpha_{8k+1}+(2i+1)\beta_{8k+2})f_{8i+8k+4}
-3\beta_{8k+3}f_{8i+8k+5}-3\alpha_{8k+4}f_{8i+8k+7}\\[3mm]
            &\hfill +3\beta_{8k+7}f_{8i+8k+9}-3\alpha_{8k+8}f_{8i+8k+11}),\\[3mm]
d(f_{8i+5})=&\sum\limits_{k=0}^{t}(((4i+3)\alpha_{8k+1}+(2i+1)\beta_{8k+2})f_{8i+8k+5}
-2\beta_{8k+3}f_{8i+8k+6}-\alpha_{8k+3}f_{8i+8k+7}\\[3mm]
            &\hfill+\alpha_{8k+4}f_{8i+8k+8}+\alpha_{8k+5}f_{8i+8k+9}+\alpha_{8k+8}f_{8i+8k+12}),\\[3mm]
d(f_{8i+6})=&\sum\limits_{k=0}^{t}(((4i+4)\alpha_{8k+1}+(2i+1)\beta_{8k+2})f_{8i+8k+6}
+\alpha_{8k+3}f_{8i+8k+8}-\alpha_{8k+4}f_{8i+8k+9}+\alpha_{8k+8}f_{8i+8k+13}),\\[3mm]
d(f_{8i+7})=&\sum\limits_{k=0}^{t}(((4i+3)\alpha_{8k+1}+(2i+2)\beta_{8k+2})f_{8i+8k+7}
+\beta_{8k+3}f_{8i+8k+8}+2\alpha_{8k+4}f_{8i+8k+10}-\alpha_{8k+5}f_{8i+8k+11}\\[3mm]
            &\hfill -\beta_{8k+7}f_{8i+8k+12}+\beta_{8k+8}f_{8i+8k+13}),\\[3mm]
d(f_{8i+8})=&\sum\limits_{k=0}^{t}(((4i+4)\alpha_{8k+1}+(2i+2)\beta_{8k+2})f_{8k+8}
-\beta_{8k+3}f_{8i+8k+9}-2\alpha_{8k+3}f_{8i+8k+10}-\alpha_{8k+4}f_{8i+8k+11}\\[3mm]
            &\hfill +\beta_{8k+7}f_{8i+8k+13}-2\beta_{8k+8}f_{8i+8k+14}-\alpha_{8k+8}f_{8i+8k+15}),\\[3mm]
\end{array}$

\hfill \emph{where}  $i\in \mathbb{N}.$

\end{prop}

From the propositions above we can immediately obtain the following corollary on the number of nil-independent derivations of the algebras  $\mathbf{n}_1$ and $\mathbf{n}_2$.
\begin{cor}
The maximal number of potentially nil-independent derivations of $\mathbf{n}_1$ and $\mathbf{n}_2$ is $2$.
\end{cor}

We denote by ${M}_1$ and $M_2$ pro-solvable Lie algebras with the maximal pro-nilpotent ideals $\mathbf{n}_1$ and $\mathbf{n}_2$, respectively. The complementary subspaces of ${M}_1$ and $M_2$ to $\mathbf{n}_1$ and $\mathbf{n}_2$ are denoted by $Q_1$ and $Q_2$, respectively.

\begin{lem}
The derivations $ad_{x}$ and $ad_y$ are non-potentially nilpotent for any $x\in Q_1$ and $y\in Q_2$, respectively.
\end{lem}
\begin{proof} Let $x\in Q_1$ and $ad_{x}$ is potentially nilpotent, i.e., $\cap_{k=1}^{\infty}Im\empty\ ad_{x}^{k}=0$. Set $M=\mathbf{n}_1+\mathbb{F}x$.
From Proposition \ref{prop1} we conclude that $ad_{x}=d$ for some $d\in Der(\mathbf{n}_1)$. The condition $\cap_{k=1}^{\infty}Im\empty\ ad_{x}^{k}=0$
implies that $\alpha_{1}=\beta_{2}=0$. Hance, we have

$$\begin{array}{lll} [x,e_{3i-2}] & = & -\alpha_3e_{3i}-\sum\limits_{k=2}^{t}(((i-1)\beta_{3k-1}+i\alpha_{3k-2})e_{3k+3i-5}
+\alpha_{3k}e_{3k+3i-3}),\\
\left[x,e_{3i-1}\right]&=&-\beta_3e_{3i}-\sum\limits_{k=2}^{t}(i\beta_{3k-1}+ (i-1)\alpha_{3k-2})e_{3k+3i-4}+\beta_{3k}e_{3k+3i-3}),\\
\left[x,e_{3i}\right]& = & \beta_3e_{3i+1}+\alpha_3e_{3i+2}-\sum\limits_{k=2}^{t}(i(\beta_{3k-1}+\alpha_{3k-2})e_{3k+3i-3}
-\beta_{3k}e_{3k+3i-2}-\alpha_{3k}e_{3k+3i-1}),\\
\end{array}$$ where  $i\in \mathbb{N}.$

One can easily verify that $\cap_{i=1}^{\infty}(M)^{i}=0$. This implies that $M$ is potentially nilpotent which is a contraduction and it shows that the assumption is wrong.

The same manner one can prove the non-potentially nilpotentness of $ad_y$ for $y\in Q_2$ referring to Proposition \ref{prop3}.
\end{proof}

From the lemma above we obtain an immidiate corollary on maximal dimensions of the subspaces $Q_1$ va $Q_2$.

\begin{cor}
The dimensions of $Q_1$ and $Q_2$ are not greater than the maximal number of potentially nil-independent derivations of
$\mathbf{n}_1$ and $\mathbf{n}_2$, respectively.
\end{cor}

In the theorem below we describe pro-solvable algebras with maximal pro-nilpotent ideal $\mathbf{n}_1$. Evidently, in this case $dimQ_1=2$.

\begin{thm} Let $M$ be a maximal potential solvable Lie algebra whose maximal potential nilpotent ideal is $\mathbf{n}_1$. Then there is
a basis $\{x,y,e_{1},e_{2},...\}$, such that the multiplication table of $M$ in this basis is given as follows
$$M=R_{\mathbf{n}_1}(\alpha):=\left\{\begin{array}{lll}
[e_{3i-2},x]&=&ie_{3i-2}+\sum\limits_{k=2}^{t}(i-1)\alpha_{k}e_{3k+3i-5},\\[1mm]
[e_{3i-1},x]&=&(i-1)e_{3i-1}+\sum\limits_{k=2}^{t}i\alpha_{k}e_{3k+3i-4},\\[1mm]
[e_{3i},x]&=&ie_{3i}+\sum\limits_{k=2}^{t}i\alpha_{k}e_{3k+3i-3},\\[1mm]
[e_{3i-2},y]&=&-e_{3i-2}+\sum\limits_{k=2}^{t}\alpha_{k}e_{3k+3i-5}\\[1mm]
[e_{3i-1},y]&=&e_{3i-1}-\sum\limits_{k=2}^{t}\alpha_{k}e_{3k+3i-4},\\[1mm]
[x,y]&=&\sum\limits_{k=2}^{t}(1-k)\alpha_{k}e_{3k}.\\[1mm]
\end{array}\right.$$
where $\alpha_{3i-2}\alpha_{3k-2}=0,\quad 2\leq i,k\leq t,\ 3k+3i-4\geq 3t,\ i,t \in \mathbb{N}, \quad \alpha=\{\alpha_2,\dots,\alpha_t\}$.
\end{thm}
\begin{proof} By using Proposition \ref{prop1} for $i\in \mathbb{N}$ we get
$$\begin{array}{lll}
[e_{3i-2},x]&=&ie_{3i-2}+\alpha_3e_{3i}+\sum\limits_{k=2}^{t}(((i-1)\beta_{3k-1}+i\alpha_{3k-2})e_{3k+3i-5}+\alpha_{3k}e_{3k+3i-3}),\\[1mm]
[e_{3i-1},x]&=&(i-1)e_{3i-1}+\beta_3e_{3i}+\sum\limits_{k=2}^{t}((i\beta_{3k-1}+ (i-1)\alpha_{3k-2})e_{3k+3i-4}+\beta_{3k}e_{3k+3i-3}),\\[1mm]
[e_{3i},x]&=&ie_{3i}-\beta_3e_{3i+1}-\alpha_3e_{3i+2}+\sum\limits_{k=2}^{t}(i(\beta_{3k-1}+
\alpha_{3k-2})e_{3k+3i-3}-\beta_{3k}e_{3k+3i-2}-\alpha_{3k}e_{3k+3i-1}),\\[1mm]
[e_{3i-2},y]&=&(i-1)e_{3i-2}+\alpha_3'e_{3i}+\sum\limits_{k=2}^{t}(((i-1)\beta_{3k-1}'+i\alpha_{3k-2}')e_{3k+3i-5}+\alpha_{3k}'e_{3k+3i-3}),\\[1mm]
[e_{3i-1},y]&=&ie_{3i-1}+\beta_3'e_{3i}+\sum\limits_{k=2}^{t}((i\beta_{3k-1}'+ (i-1)\alpha_{3k-2}')e_{3k+3i-4}+\beta_{3k}'e_{3k+3i-3}),\\[1mm]
[e_{3i},y]&=&ie_{3i}-\beta_3'e_{3i+1}-\alpha_3'e_{3i+2}+\sum\limits_{k=2}^{t}(i(\beta_{3k-1}'+
\alpha_{3k-2}')e_{3k+3i-3}-\beta_{3k}'e_{3k+3i-2}-\alpha_{3k}'e_{3k+3i-1}),\\[1mm]
[x,y]&=&\sum\limits_{k=1}^{s}\gamma_{3k-2}e_{3k-2}+\sum\limits_{k=1}^{s}\gamma_{3k-1}e_{3k-1}
+\sum\limits_{k=1}^{s}\gamma_{3k}e_{3k}+\delta_1x+\delta_2y.\\[1mm]
\end{array}$$

Let us consider the following base change
$$x'=x-\sum\limits_{k=2}^{t}\alpha_{3k-2}e_{3k-3}+\sum\limits_{k=1}^{t}(-\beta_{3k}e_{3k-2}+\alpha_{3k}e_{3k-1}), \quad
y'=y+\sum\limits_{k=2}^{t}\beta_{3k-1}'e_{3k-3}+\sum\limits_{k=1}^{t}(-\beta_{3k}'e_{3k-2}+\alpha_{3k}'e_{3k-1}).$$

Introduce the notations $\alpha_{3k-2}:=\beta_{3k-1}+ \alpha_{3k-2},\ \ \beta_{3k-1}:=\beta_{3k-1}'+ \alpha_{3k-2}'$.

On the one hand one has
$$\begin{array}{lll} [e_1,[x,y]]&=&[[e_1,x],y]-[[e_1,y],x]=[e_{1},y]-[\sum\limits_{i=2}^{t}\beta_{3i-1}e_{3i-2},x]\\
&=&\sum\limits_{i=2}^{t}\beta_{3i-1}e_{3i-2}-\sum\limits_{i=2}^{t}\beta_{3i-1}\Big(ie_{3i-2}
+\sum\limits_{k=2}^{t}(i-1)\alpha_{3k-2}e_{3k+3i-5}\Big)\\
&=&\sum\limits_{i=2}^{t}(1-i)\beta_{3i-1}e_{3i-2}-\sum\limits_{i=2}^{t}\sum\limits_{k=2}^{t}(i-1)
\beta_{3i-1}\alpha_{3k-2}e_{3k+3i-5},\\
\end{array}$$
From the other hand we get
$$\begin{array}{lll}[e_1,[x,y]]&=&[e_1,\sum\limits_{k=1}^{t}\gamma_{3k-2}e_{3k-2}
+\sum\limits_{k=1}^{t}\gamma_{3k-1}e_{3k-1}+\sum\limits_{k=1}^{t}\gamma_{3k}e_{3k}+\delta_1x+\delta_2y]\\
&=&\sum\limits_{k=1}^{t}\gamma_{3k-1}e_{3k}+\sum\limits_{k=1}^{s}\gamma_{3k}e_{3k+1}+\delta_1e_{1}
+\delta_2\sum\limits_{i=2}^{t}\beta_{3i-1}e_{3i-2}.\\
\end{array}$$
Writing the Jacobi identity for the vectors $\{e_2,x,y\}$ we obtain
$$\gamma_{3k-2}=\gamma_{3k-1}=\delta_1=\delta_2=0.$$

Let now consiider the identity
$$\begin{array}{lll}0&=&[e_3,[x,y]]=[[e_3,x],y]-[[e_3,y],x]=[e_{3}+\sum\limits_{i=2}^{t}\alpha_{3i-2}e_{3i},y]-[e_{3}
+\sum\limits_{i=2}^{t}\beta_{3i-1}e_{3i},x]\\
&=&e_{3}+\sum\limits_{i=2}^{t}\beta_{3i-1}e_{3i}+\sum\limits_{i=2}^{t}\alpha_{3i-2}\Big(ie_{3i}
+\sum\limits_{k=2}^{t}i\beta_{3k-1}e_{3k+3i-3}\Big)\\
&&\hfill -e_{3}-\sum\limits_{i=2}^{t}\alpha_{3i-2}e_{3i}-\sum\limits_{i=2}^{t}\beta_{3i-1}\Big(ie_{3i}
+\sum\limits_{k=2}^{t}i\alpha_{3k-2}e_{3k+3i-3}\Big)\\
&=&\sum\limits_{i=2}^{t}(1-i)(\beta_{3i-1}-\alpha_{3i-2})e_{3i}
+\sum\limits_{i=2}^{t}\sum\limits_{k=2}^{t}i\Big(\alpha_{3i-2}\beta_{3k-1}
-\beta_{3i-1}\alpha_{3k-2}\Big)e_{3k+3i-3}.\end{array}$$

From where we get $\beta_{3i-1}=\alpha_{3i-2}$.

Thus we obtain
$$\gamma_{3i}=(1-i)\alpha_{3i-2}, \quad \alpha_{3i-2}\alpha_{3k-2}=0,\quad 2\leq i,k\leq t,\ 3k+3i-4\geq 3t,$$
Taking into account all the above and considering the base change $y'=y-x$ we come to $R_{\mathbf{n}_1}(\alpha)$.
\end{proof}

Pro-solvable Lie algebras with maximal pro-nilpotent ideal is $\mathbf{n}_2$ can be described similarly. In this case also $dim Q_2=2$. The result is given by the following theorem.

\begin{thm} Let $M$ be an maximal potential solvable Lie algebra whose maximal by including potential nilpotent ideal is $\mathbf{n}_2$. Then there is a basis $\{x,y,f_{1},f_{2},...\}$ such that the multiplication table of $M$ on this basis is given as follows
$$M=R_{\mathbf{n}_2}(\beta):=\left\{\begin{array}{lll}
[f_{8i+1},x]=f_{8i+1},\\[1mm]
[f_{8i+2},x]=-2f_{8i+2},\\[1mm]
[f_{8i+3},x]=-f_{8i+3},\\[1mm]
[f_{8i+5},x]=f_{8i+5},\\[1mm]
[f_{8i+6},x]=2f_{8i+6},\\[1mm]
[f_{8i+7},x]=-f_{8i+7}.\\[1mm]
\end{array} \quad
\begin{array}{lll}
[f_{8i+1},y]=&2if_{8i+1}+\sum\limits_{k=1}^{t}2i\beta_{k}f_{8i+8k+1},\\[1mm]
[f_{8i+2},y]=&(2i+1)f_{8i+2}+\sum\limits_{k=1}^{t}(2i+1)\beta_{k}f_{8i+8k+2},\\[1mm]
[f_{8i+3},y]=&(2i+1)f_{8i+3}+\sum\limits_{k=1}^{t}(2i+1)\beta_{k}f_{8i+8k+3},\\[1mm]
[f_{8i+4},y]=&(2i+1)f_{8i+4}+\sum\limits_{k=1}^{t}(2i+1)\beta_{k}f_{8i+8k+4},\\[1mm]
[f_{8i+5},y]=&(2i+1)f_{8i+5}+\sum\limits_{k=1}^{t}(2i+1)\beta_{k}f_{8i+8k+5},\\[1mm]
[f_{8i+6},y]=&(2i+1)f_{8i+6}+\sum\limits_{k=1}^{t}(2i+1)\beta_{k}f_{8i+8k+6},\\[1mm]
[f_{8i+7},y]=&(2i+2)f_{8i+7}+\sum\limits_{k=1}^{t}(2i+2)\beta_{k}f_{8i+8k+7},\\[1mm]
[f_{8i+8},y]=&(2i+2)f_{8i+8}+\sum\limits_{k=1}^{t}(2i+2)\beta_{k}f_{8i+8k+8},\\[1mm]
\end{array}\right.$$
\hfill where $i,t\in \mathbb{N}, \quad \beta=\{\beta_1,\dots,\beta_t\}.$
\end{thm}
\begin{proof} We use Proposition \ref{prop3} to get the following equalities
{\small
$$\begin{array}{lll}
[f_{8i+1},x]=&(4i+1)f_{8i+1}+\sum\limits_{k=1}^{t}((4i+1)\alpha_{1,8k+1}+2i\beta_{1,8k+2})f_{8i+8k+1}
+\sum\limits_{k=0}^{t}(\alpha_{1,8k+3}f_{8i+8k+3}\\[1mm]
             &\hfill +\alpha_{1,8k+4}f_{8i+8k+4}+\alpha_{1,8k+5}f_{8i+8k+5}-2\beta_{1,8k+7}f_{8i+8k+6}+\alpha_{1,8k+8}f_{8i+8k+8}),\\[1mm]
[f_{8i+2},x]=&4if_{8i+2}+\sum\limits_{k=1}^{t}(4i\alpha_{1,8k+1}+(2i+1)\beta_{1,8k+2})f_{8i+8k+2}\\[1mm]
             &\hfill+\sum\limits_{k=0}^{t}(\beta_{1,8k+3}f_{8i+8k+3}+\beta_{1,8k+7}f_{8i+8k+7}+\beta_{1,8k+8}f_{8i+8k+8}).\\[1mm]
[f_{8i+3},x]=&(4i+1)f_{8i+3}+\sum\limits_{k=1}^{t}((4i+1)\alpha_{1,8k+1}+(2i+1)\beta_{1,8k+2})f_{8i+8k+3}\\[1mm]
             &\hfill+\sum\limits_{k=0}^{t}(\beta_{1,8k+3}f_{8i+8k+4}-\alpha_{1,8k+5}f_{8i+8k+7}-\beta_{1,8k+7}f_{8i+8k+8}-\beta_{1,8k+8}f_{8i+8k+9}-2\alpha_{1,8k+8}f_{8i+8k+10}).\\[1mm]
[f_{8i+4},x]=&(4i+2)f_{8i+4}+\sum\limits_{k=1}^{t}((4i+2)\alpha_{1,8k+1}+(2i+1)\beta_{1,8k+2})f_{8i+8k+4}\\[1mm]
             &\hfill+\sum\limits_{k=0}^{t}(-3\beta_{1,8k+3}f_{8i+8k+5}-3\alpha_{1,8k+4}f_{8i+8k+7}+3\beta_{1,8k+7}f_{8i+8k+9}-3\alpha_{1,8k+8}f_{8i+8k+11}).\\[1mm]
[f_{8i+5},x]=&(4i+3)f_{8i+5}+\sum\limits_{k=1}^{t}((4i+3)\alpha_{1,8k+1}+(2i+1)\beta_{1,8k+2})f_{8i+8k+5}
+\sum\limits_{k=0}^{t}(-2\beta_{1,8k+3}f_{8i+8k+6}\\[1mm]
             &\hfill-\alpha_{1,8k+3}f_{8i+8k+7}+\alpha_{1,8k+4}f_{8i+8k+8}+\alpha_{1,8k+5}f_{8i+8k+9}+\alpha_{1,8k+8}f_{8i+8k+12}).\\[1mm]
[f_{8i+6},x]=&(4i+4)f_{8i+6}+\sum\limits_{k=1}^{t}((4i+4)\alpha_{1,8k+1}+(2i+1)\beta_{1,8k+2})f_{8i+8k+6}\\[1mm]
             &\hfill+\sum\limits_{k=0}^{t}(\alpha_{1,8k+3}f_{8i+8k+8}-\alpha_{1,8k+4}f_{8i+8k+9}+\alpha_{1,8k+8}f_{8i+8k+13}).\\[1mm]
[f_{8i+7},x]=&(4i+3)f_{8i+7}+\sum\limits_{k=1}^{t}((4i+3)\alpha_{1,8k+1}+(2i+2)\beta_{1,8k+2})f_{8i+8k+7}
+\sum\limits_{k=0}^{t}(\beta_{1,8k+3}f_{8i+8k+8}\\[1mm]
             &\hfill+2\alpha_{1,8k+4}f_{8i+8k+10}-\alpha_{1,8k+5}f_{8i+8k+11}-\beta_{1,8k+7}f_{8i+8k+12}+\beta_{1,8k+8}f_{8i+8k+13}).\\[1mm]
[f_{8i+8},x]=&(4i+4)f_{8i+8}+\sum\limits_{k=0}^{t}((4i+4)\alpha_{1,8k+1}+(2i+2)\beta_{1,8k+2})f_{8i+8k+8}
+\sum\limits_{k=0}^{t}(-\beta_{1,8k+3}f_{8i+8k+9}\\[1mm]
             &\hfill-2\alpha_{1,8k+3}f_{8i+8k+10}-\alpha_{1,8k+4}f_{8i+8k+11}+\beta_{1,8k+7}f_{8i+8k+13}
             -2\beta_{1,8k+8}f_{8i+8k+14}-\alpha_{1,8k+8}f_{8i+8k+15}).\\[1mm]
\end{array}$$}
{\small
$$\begin{array}{lll}
[f_{8i+1},y]=&2if_{8i+1}+\sum\limits_{k=1}^{t}((4i+1)\alpha_{2,8k+1}+2i\beta_{2,8k+2})f_{8i+8k+1}
+\sum\limits_{k=0}^{t}(\alpha_{2,8k+3}f_{8i+8k+3}\\[1mm]
             &\hfill+\alpha_{2,8k+4}f_{8i+8k+4}+\alpha_{2,8k+5}f_{8i+8k+5}-2\beta_{2,8k+7}f_{8i+8k+6}+\alpha_{2,8k+8}f_{8i+8k+8}),\\[1mm]
[f_{8i+2},y]=&(2i+1)f_{8i+2}+\sum\limits_{k=1}^{t}(4i\alpha_{2,8k+1}+(2i+1)\beta_{2,8k+2})f_{8i+8k+2}\\[1mm]
             &\hfill+\sum\limits_{k=0}^{t}(\beta_{2,8k+3}f_{8i+8k+3}+\beta_{2,8k+7}f_{8i+8k+7}
             +\beta_{2,8k+8}f_{8i+8k+8}).\\[1mm]
[f_{8i+3},y]=&(2i+1)f_{8i+3}+\sum\limits_{k=1}^{t}((4i+1)\alpha_{2,8k+1}+(2i+1)\beta_{2,8k+2})f_{8i+8k+3}\\[1mm]
             &\hfill+\sum\limits_{k=0}^{t}(\beta_{2,8k+3}f_{8i+8k+4}-\alpha_{2,8k+5}f_{8i+8k+7}-\beta_{2,8k+7}f_{8i+8k+8}-\beta_{2,8k+8}f_{8i+8k+9}-2\alpha_{2,8k+8}f_{8i+8k+10}).\\[1mm]
[f_{8i+4},y]=&(2i+1)f_{8i+4}+\sum\limits_{k=1}^{t}((4i+2)\alpha_{2,8k+1}+(2i+1)\beta_{2,8k+2})f_{8i+8k+4}\\[1mm]
             &\hfill+\sum\limits_{k=0}^{t}(-3\beta_{2,8k+3}f_{8i+8k+5}-3\alpha_{2,8k+4}f_{8i+8k+7}+3\beta_{2,8k+7}f_{8i+8k+9}-3\alpha_{2,8k+8}f_{8i+8k+11}).\\[1mm]
[f_{8i+5},y]=&(2i+1)f_{8i+5}+\sum\limits_{k=1}^{t}((4i+3)\alpha_{2,8k+1}+(2i+1)\beta_{2,8k+2})f_{8i+8k+5}
+\sum\limits_{k=0}^{t}(-2\beta_{2,8k+3}f_{8i+8k+6}\\[1mm]
             &\hfill-\alpha_{2,8k+3}f_{8i+8k+7}+\alpha_{2,8k+4}f_{8i+8k+8}+\alpha_{2,8k+5}f_{8i+8k+9}+\alpha_{2,8k+8}f_{8i+8k+12}).\\[1mm]
[f_{8i+6},y]=&(2i+1)f_{8i+6}+\sum\limits_{k=1}^{t}((4i+4)\alpha_{2,8k+1}+(2i+1)\beta_{2,8k+2})f_{8i+8k+6}\\[1mm]
             &\hfill+\sum\limits_{k=0}^{t}(\alpha_{2,8k+3}f_{8i+8k+8}-\alpha_{2,8k+4}f_{8i+8k+9}+\alpha_{2,8k+8}f_{8i+8k+13}).\\[1mm]
[f_{8i+7},y]=&(2i+2)f_{8i+7}+\sum\limits_{k=1}^{t}((4i+3)\alpha_{2,8k+1}+(2i+2)\beta_{2,8k+2})f_{8i+8k+7}
+\sum\limits_{k=0}^{t}(\beta_{2,8k+3}f_{8i+8k+8}\\[1mm]
             &\hfill+2\alpha_{2,8k+4}f_{8i+8k+10}-\alpha_{2,8k+5}f_{8i+8k+11}-\beta_{2,8k+7}f_{8i+8k+12}+\beta_{2,8k+8}f_{8i+8k+13}).\\[1mm]
\end{array}$$}
{\small
$$\begin{array}{lll}
[f_{8i+8},y]=&(2i+2)f_{8i+8}+\sum\limits_{k=0}^{t}((4i+4)\alpha_{2,8k+1}+(2i+2)\beta_{2,8k+2})f_{8i+8k+8}
+\sum\limits_{k=0}^{t}(-\beta_{2,8k+3}f_{8i+8k+9}\\[1mm]
             &\hfill-2\alpha_{2,8k+3}f_{8i+8k+10}-\alpha_{2,8k+4}f_{8i+8k+11}+\beta_{2,8k+7}f_{8i+8k+13}-2\beta_{2,8k+8}f_{8i+8k+14}-\alpha_{2,8k+8}f_{8i+8k+15}).\\[1mm]
\quad \ \ [x,y]=&\sum\limits_{k=0}^{t}(\alpha_{8k+1}f_{8k+1}+\alpha_{8k+2}f_{8k+2}+\alpha_{8k+3}f_{8k+3}
+\alpha_{8k+4}f_{8k+4}+\alpha_{8k+5}f_{8k+5}\\[1mm]
      &\hfill+\alpha_{8k+6}f_{8k+6}+\alpha_{8k+7}f_{8k+7}+\alpha_{8k+8}f_{8k+8})+\delta_1x+\delta_2y,\\[1mm]
\end{array}$$} \hfill where $i,t\in \mathbb{N}$.

Consider the following base change
$$\begin{array}{lll} x'&=&x+\sum\limits_{k=1}^{t}\alpha_{1,8k+1}f_{8k}+\sum\limits_{k=0}^{t}(\beta_{1,8k+3}f_{8k+1}
-\alpha_{1,8k+3}f_{8k+2}-\alpha_{1,8k+4}f_{8k+3}+\frac{1}3\alpha_{1,8k+5}f_{8k+4}\\
&& \hfill -\beta_{1,8k+7}f_{8k+5}+\beta_{1,8k+8}f_{8k+6}-\alpha_{1,8k+8}f_{8k+7}),\\
y'&=&y+\sum\limits_{k=1}^{t}\alpha_{2,8k+1}f_{8k}+\sum\limits_{k=0}^{t}(\beta_{2,8k+3}f_{8k+1}
-\alpha_{2,8k+3}f_{8k+2}-\alpha_{2,8k+4}f_{8k+3}+\frac{1}3\alpha_{2,8k+5}f_{8k+4}\\
&& \hfill -\beta_{2,8k+7}f_{8k+5}+\beta_{2,8k+8}f_{8k+6}-\alpha_{2,8k+8}f_{8k+7}).\end{array}$$

Then we get $\beta_{j,8k+3}=\alpha_{j,8k+3}=\alpha_{j,8k+4}=\alpha_{j,8k+5}=\beta_{j,8k+7}=\alpha_{j,8k+8}=\beta_{j,8k+8}=0,\ 1\leq j\leq 2$. \\
Introduce the notations
$$\beta_{1,8k}:=2\alpha_{1,8k+1}+\beta_{1,8k+2},\quad  \quad  \beta_{2,8k}:=2\alpha_{2,8k+1}+\beta_{2,8k+2},$$

Consider the identity
$$[f_1,[x,y]]=[[f_1,x],y]-[[f_1,y],x]=[f_{1},y]=0.$$
From the other hand we can compute it as follows
$$\begin{array}{lll}[f_1,[x,y]]&=&[f_1,\sum\limits_{k=0}^{t}(\alpha_{8k+1}f_{8k+1}+\alpha_{8k+2}f_{8k+2}
+\alpha_{8k+3}f_{8k+3}+\alpha_{8k+4}f_{8k+4}+\alpha_{8k+5}f_{8k+5}\\
&&\hfill+\alpha_{8k+6}f_{8k+6}+\alpha_{8k+7}f_{8k+7}+\alpha_{8k+8}f_{8k+8})+\delta_1x+\delta_2y]\\
&=&\sum\limits_{k=0}^{t}(\alpha_{8k+2}f_{8k+3}+\alpha_{8k+3}f_{8k+4}
-3\alpha_{8k+4}f_{8k+5}-2\alpha_{8k+5}f_{8k+6}+\alpha_{8k+7}f_{8k+8}-\alpha_{8k+8}f_{8k+9})+\delta_1f_1.\end{array}$$
Thus we get
$$\alpha_{8k+2}=\alpha_{8k+3}=\alpha_{8k+4}=\alpha_{8k+5}=\alpha_{8k+7}=\alpha_{8k+8}=\delta_1=0.$$

Let us now consider the identity
$$\begin{array}{lll} [f_2,[x,y]]&=&[[f_2,x],y]-[[f_2,y],x]\\
&=&\left[\sum\limits_{i=1}^{t}\beta_{1,8i}f_{8i+2},y]-[f_{2}+\sum\limits_{i=1}^{t}\beta_{2,8i}f_{8i+2},x\right]\\
&=&\sum\limits_{i=1}^{t}\beta_{1,8i}\Big((2i+1)f_{8i+2}+\sum\limits_{k=1}^{t}(2i+1)\beta_{2,8k}f_{8i+8k+2}\Big)\\
&&\hfill -\sum\limits_{i=1}^{t}\beta_{1,8i}f_{8i+2}-\sum\limits_{i=1}^{t}\beta_{2,8i}\Big(4if_{8i+2}
+\sum\limits_{k=1}^{t}(2i+1)\beta_{1,8k}f_{8i+8k+2}\Big)\\
&=&\sum\limits_{i=1}^{t}2i(\beta_{1,8i}-2\beta_{2,8i})f_{8i+2}
+\sum\limits_{i=1}^{t}\Big(\sum\limits_{k=1}^{t}(2i+1)(\beta_{1,8i}\beta_{2,8k}
-\beta_{1,8k}\beta_{2,8i})f_{8i+8k+2}\Big).\end{array}$$
We can compute $[f_2,[x,y]]$ differently as follows
$$\begin{array}{lll}[f_2,[x,y]]&=&[f_2,\sum\limits_{k=0}^{t}(\alpha_{8k+1}f_{8k+1}+\alpha_{8k+6}f_{8k+6})+\delta_2y]\\
&=&\sum\limits_{k=0}^{t}(-\alpha_{8k+1}f_{8k+3}-\alpha_{8k+6}f_{8k+8})+\delta_2(f_{2}
+\sum\limits_{i=1}^{t}\beta_{2,8i+2}f_{8i+2}).\end{array}$$
Combining we obtain
$$\beta_{1,8k}=2\beta_{2,8k},\quad \alpha_{8k+1}=\alpha_{8k+6}=\delta_2=0.$$

Now we apply the base change $x'=x-2y$ to get the algebra $R_{\mathbf{n}_2}(\beta)$.
\end{proof}

\subsection{Low-dimensional (co)homology groups of $R_{\mathbf{n}_1}(0)$ and $R_{\mathbf{n}_2}(0)$}

\

\

In this section we compute low-dimensoinal cohomology groups of the algebras $R_{\mathbf{n}_1}(0)$ and $R_{\mathbf{n}_2}(0)$. It seemed that the cohomology groups ${\rm H}^1(L,L)$ and ${\rm H}^2(L,L)$, where $L=R_{\mathbf{n}_1}(0), R_{\mathbf{n}_2}(0)$, are trivial and the algebras are complete. To do this we first describe the derivations of $R_{\mathbf{n}_1}(0)$ and $R_{\mathbf{n}_2}(0)$.

\begin{prop}\label{prop2}
The derivations of the algebra $R_{\mathbf{n}_1}(0)$ on the basis vectors $\{x, y, e_1, e_2,... \}$ are given as follows
$$\begin{array}{lll} d(e_{3i-2})&=&(i\alpha_{1}+(i-1)\beta_2)e_{3i-2}+\sum\limits_{k=2}^{t}\alpha_{3k-2}e_{3k+3i-5}
+\sum\limits_{k=1}^{t}\alpha_{3k}e_{3k+3i-3},\\
d(e_{3i-1})&=&((i-1)\alpha_1+i\beta_{2})e_{3i-1}-\sum\limits_{k=2}^{t}\alpha_{3k-2}e_{3k+3i-4}
+\sum\limits_{k=1}^{t}\beta_{3k}e_{3k+3i-3},\\
d(e_{3i})&=&i(\beta_{2}+ \alpha_{1})e_{3i}-\sum\limits_{k=1}^{t}(\beta_{3k}e_{3k+3i-2}+\alpha_{3k}e_{3k+3i-1}),\\
d(x)&=&-\sum\limits_{k=1}^{t}k\beta_{3k}e_{3k-2}+\sum\limits_{k=2}^{t}(k-1)\alpha_{3k}e_{3k-1}
+\sum\limits_{k=2}^{t}(1-k)\alpha_{3k-2}e_{3k-3}\\
d(y)&=&\sum\limits_{k=1}^{t}\beta_{3k}e_{3k-2}+\sum\limits_{k=1}^{t}\alpha_{3k}e_{3k-1}.\\
&&\hfill where \ i,t \in \mathbb{N}.\\
\end{array}$$
\end{prop}
\begin{proof} Due to $[\mathbf{n}_1,Q_1]=\mathbf{n}_1$ any derivation $d \in DerR_{\mathbf{n}_1}(0)$ is a derivation of $\mathbf{n}_1$
$$d(\mathbf{n}_1)=d([\mathbf{n}_1,Q_1])=[d(\mathbf{n}_1),Q_1]+[\mathbf{n}_1,d(Q_1)]\subset \mathbf{n}_1.$$
We use Proposition \ref{prop1} to get

$$\begin{array}{lll}d(e_{3i-2})&=&\sum\limits_{k=1}^{t}(((i-1)\beta_{3k-1}+i\alpha_{3k-2})e_{3k+3i-5}
+\alpha_{3k}e_{3k+3i-3}),\\
d(e_{3i-1})&=&\sum\limits_{k=1}^{t}(i\beta_{3k-1}+ (i-1)\alpha_{3k-2})e_{3k+3i-4}+\beta_{3k}e_{3k+3i-3}),\\
d(e_{3i})&=&\sum\limits_{k=1}^{t}(i(\beta_{3k-1}+ \alpha_{3k-2})e_{3k+3i-3}-\beta_{3k}e_{3k+3i-2}-\alpha_{3k}e_{3k+3i-1}),\end{array}$$ \hfill \emph{where} $i \in \mathbb{N}$

and
let
$$d(x)=\sum\limits_{k=1}^{t}\gamma_{3k-2}e_{3k-2}+\sum\limits_{k=1}^{t}\gamma_{3k-1}e_{3k-1}+\sum\limits_{k=1}^{t}\gamma_{3k}e_{3k}+b_{1,1}x+b_{1,2}y,$$
$$d(y)=\sum\limits_{k=1}^{t}\delta_{3k-2}e_{3k-2}+\sum\limits_{k=1}^{t}\delta_{3k-1}e_{3k-1}
+\sum\limits_{k=1}^{t}\delta_{3k}e_{3k}+b_{2,1}x+b_{2,2}y.$$

Applying the derivation $d$ to $[e_1,y]$ we get
$$\begin{array}{lll}d([e_1,y])&=&[d(e_1),y]+[e_1,d(y)]\\
&=&\left[\sum\limits_{k=1}^{t}(\alpha_{3k-2}e_{3k-2}+\alpha_{3k}e_{3k}),y\right]
+\left[e_1,\sum\limits_{k=1}^{t}\delta_{3k-2}e_{3k-2}+\sum\limits_{k=1}^{t}\delta_{3k-1}e_{3k-1}\right.\\
&&\hfill \left.+\sum\limits_{k=1}^{t}\delta_{3k}e_{3k}+b_{2,1}x+b_{2,2}y\right]\\
&=&-\sum\limits_{k=1}^{t}\alpha_{3k-2}e_{3k-2}-\sum\limits_{k=1}^{t}\delta_{3k-1}e_{3k}
+\sum\limits_{k=1}^{t}\delta_{3k}e_{3k+1}+b_{2,1}e_1-b_{2,2}e_1.\end{array}$$
On the other hand
$$d([e_1,y])=d(-e_1)=-\sum\limits_{k=1}^{t}(\alpha_{3k-2}e_{3k-2}+\alpha_{3k}e_{3k}).$$
Therefore we obtain
$$b_{2,1}=b_{2,2},\quad  \delta_{3k-1}=\alpha_{3k},\quad \delta_{3k}=0,\quad 1\leq k\leq t.$$
Let us apply $d$ to $[e_2,y]$ then one has
$$\begin{array}{lll}d([e_2,y])&=&[d(e_2),y]+[e_2,d(y)]\\
&=&\left[\sum\limits_{k=1}^{t}(\beta_{3k-1}e_{3k-1}+\beta_{3k}e_{3k}),y\right]
+\left[e_2,\sum\limits_{k=1}^{t}\delta_{3k-2}e_{3k-2}+\sum\limits_{k=1}^{t}\alpha_{3k}e_{3k-1}+b_{2,2}x+b_{2,2}y\right]\\
&=&\sum\limits_{k=1}^{t}\beta_{3k-1}e_{3k-1}+\sum\limits_{k=1}^{t}\delta_{3k-2}e_{3k}+b_{2,2}e_2\end{array}$$
and computing it another way\\
$$d([e_2,y])=d(e_2)=\sum\limits_{k=1}^{t}(\beta_{3k-1}e_{3k-1}+\beta_{3k}e_{3k})$$
we get
$$b_{2,2}=0,\quad  \delta_{3k-2}=\beta_{3k},\quad 1\leq k\leq t.$$
Therefore,
$$d(y)=\sum\limits_{k=1}^{t}\beta_{3k}e_{3k-2}+\sum\limits_{k=1}^{t}\alpha_{3k}e_{3k-1},\quad  i\in N,$$
Now we consider
$$\begin{array}{lll}d([e_1,x])&=&[d(e_1),x]+[e_1,d(x)]\\
&=&\left[\sum\limits_{k=1}^{t}(\alpha_{3k-2}e_{3k-2}+\alpha_{3k}e_{3k}),x\right]\\
&&\hfill +[e_1,\sum\limits_{k=1}^{t}\gamma_{3k-2}e_{3k-2}+\sum\limits_{k=1}^{t}\gamma_{3k-1}e_{3k-1}
+\sum\limits_{k=1}^{t}\gamma_{3k}e_{3k}+b_{1,1}x+b_{1,2}y]\\
&=&\sum\limits_{k=1}^{t}(k\alpha_{3k-2}e_{3k-2}+k\alpha_{3k}e_{3k})-\sum\limits_{k=1}^{t}\gamma_{3k-1}e_{3k}
+\sum\limits_{k=1}^{t}\gamma_{3k}e_{3k+1}+(b_{1,1}-b_{1,2})e_1\end{array}$$
and we also compute it another way
$$d([e_1,x])=d(e_1)=\sum\limits_{k=1}^{t}(\alpha_{3k-2}e_{3k-2}+\alpha_{3k}e_{3k}).$$
to get the equalities
$$b_{1,2}=b_{1,1},\quad \gamma_{3t}=\gamma_2=0,\quad \gamma_{3k-3}=(1-k)\alpha_{3k-2},\quad \gamma_{3k-1}=(k-1)\alpha_{3k},\quad 2\leq
k\leq t.$$
If we consider
$$\begin{array}{lll}0&=&d([e_2,x])=[d(e_2),x]+[e_2,d(x)]\\
&=&\left[\sum\limits_{k=1}^{t}(\beta_{3k-1}e_{3k-1}+\beta_{3k}e_{3k}),x\right]\\
&&\hfill+[e_2,\sum\limits_{k=1}^{t}\gamma_{3k-2}e_{3k-2}+\sum\limits_{k=2}^{t}(k-1)\alpha_{3k}e_{3k-1}
+\sum\limits_{k=2}^{t}(1-k)\alpha_{3k-2}e_{3k-3}+b_{1,1}x+b_{1,1}y]\\
&=&\sum\limits_{k=1}^{t}((k-1)\beta_{3k-1}e_{3k-1}+k\beta_{3k}e_{3k})+\sum\limits_{k=1}^{t}\gamma_{3k-2}e_{3k}
-\sum\limits_{k=2}^{t}(1-k)\alpha_{3k-2}e_{3k-1}+b_{1,1}e_2\end{array}$$
we obtain
$$b_{1,1}=0,\quad \gamma_{3k-2}=-k\beta_{3k},\ \ 1\leq k\leq t, \quad \beta_{3k-1}=-\alpha_{3k-2},\quad 2\leq k\leq t.$$
\end{proof}

The counterpart of the result above for $R_{\mathbf{n}_2}(0)$ also is true and can be proved similarly to that of Proposition \ref{prop2}. Below we give the result without proof.

\begin{prop}\label{1prop2}
The derivatives of $R_{\mathbf{n}_2}(0)$ are given as follows
$$\begin{array}{lll}
d(f_{8i+1})=&((4i+1)\alpha_{1}+2i\beta_{2})f_{8i+1}+\sum\limits_{k=1}^{t}\alpha_{8k+1}f_{8i+8k+1}
+\sum\limits_{k=0}^{t}(\alpha_{8k+3}f_{8i+8k+3}\\[2mm]
            &\hfill+\alpha_{8k+4}f_{8i+8k+4}+\alpha_{8k+5}f_{8i+8k+5}-2\beta_{8k+7}f_{8i+8k+6}+\alpha_{8k+8}f_{8i+8k+8}),\\[2mm]
d(f_{8i+2})=&(4i\alpha_{1}+(2i+1)\beta_{2})f_{8i+2}-\sum\limits_{k=1}^{t}2\alpha_{8k+1}f_{8i+8k+2}\\[2mm]
            &\hfill+\sum\limits_{k=0}^{t}(\beta_{8k+3}f_{8i+8k+3}+\beta_{8k+7}f_{8i+8k+7}+\beta_{8k+8}f_{8i+8k+8}),\\[2mm]
d(f_{8i+3})=&((4i+1)\alpha_{1}+(2i+1)\beta_2)f_{8i+3}-\sum\limits_{k=1}^{t}\alpha_{8k+1}f_{8i+8k+3}
+\sum\limits_{k=0}^{t}(\beta_{8k+3}f_{8i+8k+4}-\alpha_{8k+5}f_{8i+8k+7}\\[2mm]
            &\hfill-\beta_{8k+7}f_{8i+8k+8}-\beta_{8k+8}f_{8i+8k+9}-2\alpha_{8k+8}f_{8i+8k+10})\\[2mm]
d(f_{8i+4})=&((4i+2)\alpha_{1}+(2i+1)\beta_{2})f_{8i+4}+\sum\limits_{k=0}^{t}(-3\beta_{8k+3}f_{8i+8k+5}
-3\alpha_{8k+4}f_{8i+8k+7}\\[2mm]
            &\hfill+3\beta_{8k+7}f_{8i+8k+9}-3\alpha_{8k+8}f_{8i+8k+11}),\\[2mm]
d(f_{8i+5})=&((4i+3)\alpha_{1}+(2i+1)\beta_{2})f_{8i+5}+\sum\limits_{k=1}^{t}\alpha_{8k+1}f_{8i+8k+5}
+\sum\limits_{k=0}^{t}(-2\beta_{8k+3}f_{8i+8k+6}-\alpha_{8k+3}f_{8i+8k+7}\\[2mm]
            &\hfill+\alpha_{8k+4}f_{8i+8k+8}+\alpha_{8k+5}f_{8i+8k+9}+\alpha_{8k+8}f_{8i+8k+12}),\\[2mm]
d(f_{8i+6})=&((4i+4)\alpha_{1}+(2i+1)\beta_{2})f_{8i+6}+\sum\limits_{k=1}^{t}2\alpha_{8k+1}f_{8i+8k+6}\\[2mm]
            &\hfill+\sum\limits_{k=0}^{t}(\alpha_{8k+3}f_{8i+8k+8}-\alpha_{8k+4}f_{8i+8k+9}+\alpha_{8k+8}f_{8i+8k+13}),\\[2mm]
d(f_{8i+7})=&((4i+3)\alpha_{1}+(2i+2)\beta_{2})f_{8i+7}-\sum\limits_{k=1}^{t}\alpha_{8k+1}f_{8i+8k+7}
+\sum\limits_{k=0}^{t}(\beta_{8k+3}f_{8i+8k+8}\\[2mm]
            &\hfill+2\alpha_{8k+4}f_{8i+8k+10}-\alpha_{8k+5}f_{8i+8k+11}-\beta_{8k+7}f_{8i+8k+12}
            +\beta_{8k+8}f_{8i+8k+13}),\\[2mm]
d(f_{8i+8})=&((4i+4)\alpha_{1}+(2i+2)\beta_{2})f_{8i+8}+\sum\limits_{k=0}^{t}(-\beta_{8k+3}f_{8i+8k+9}
-2\alpha_{8k+3}f_{8i+8k+10}-\alpha_{8k+4}f_{8i+8k+11}\\[2mm]
            &\hfill+\beta_{8k+7}f_{8i+8k+13}-2\beta_{8k+8}f_{8i+8k+14}-\alpha_{8k+8}f_{8i+8k+15}),\\[2mm]
\quad \ \ d(x)=&\sum\limits_{k=0}^{t}(\beta_{8k+3}f_{8k+1}+2\alpha_{8k+3}f_{8k+2}+\alpha_{8k+4}f_{8k+3}-\beta_{8k+7}f_{8k+5}+2\beta_{8k+8}f_{8k+6}+\alpha_{8k+8}f_{8k+7}),\\[2mm]
\quad \ \ d(y)=&\sum\limits_{k=0}^{t}(2k\beta_{8k+3}f_{8k+1}-(2k+1)\alpha_{8k+3}f_{8k+2}-(2k+1)\alpha_{8k+4}f_{8k+3}
+\frac{1}{3}(2k+1)\alpha_{8k+5}f_{8k+4}\\[2mm]
     &\hfill-(2k+1)\beta_{8k+7}f_{8k+5}+(2k+1)\beta_{8k+6}f_{8k+6}-(2k+2)\alpha_{8k+8}f_{8k+7})+\sum\limits_{k=0}^{t-1}2(k+1)\alpha_{8k+9}f_{8k+8}.\\[2mm]
\end{array}$$
\hfill where $i,t\in \mathbb{N}.$
\end{prop}

Now we show that the algebra $R_{\mathbf{n}_1}(0)$ is complete.
\begin{thm}\label{thm2}
The potential solvable Lie algebra $R_{\mathbf{n}_1}(0)$ is complete.
\end{thm}

\begin{proof}
It is easy to see from the product rules of $R_{\mathbf{n}_1}(0)$ that $Center (R_{\mathbf{n}_1}(0))={0}.$
Now, we shall show that all derivations of $R_{\mathbf{n}_1}(0)$ are inner.

Let $d$ be a derivation of $R_{\mathbf{n}_1}(0)$. We show that there exists $a \in R_{\mathbf{n}_1}(0)$ such that $d\equiv ad_a.$ It is sufficient to verify this on the basis vectors $\{x, y, e_1, e_2,... \}$. We claim that
$$a=(\alpha_1+\beta_2)x+\beta_2y+\sum\limits_{k=1}^{t}\beta_{3k}e_{3k-2}
-\sum\limits_{k=1}^{t}\alpha_{3k}e_{3k-1}+\sum\limits_{k=2}^{t}\alpha_{3k-2}e_{3k-3}$$
is desired such an element of $R_{\mathbf{n}_1}(0)$ for the derivations from Proposition
\ref{prop2}. Indeed,

$$\begin{array}{lll}d(e_{3i-2})-ad_a(e_{3i-2})&=&(i\alpha_{1}+(i-1)\beta_2)e_{3i-2}
+\sum\limits_{k=2}^{t}\alpha_{3k-2}e_{3k+3i-5}+\sum\limits_{k=1}^{t}\alpha_{3k}e_{3k+3i-3}\\
&&\hfill-[e_{3i-2},(\alpha_1+\beta_2)x+\beta_2y+\sum\limits_{k=1}^{t}\beta_{3k}e_{3k-2}
-\sum\limits_{k=1}^{t}\alpha_{3k}e_{3k-1}+\sum\limits_{k=2}^{t}\alpha_{3k-2}e_{3k-3}]=0,\end{array}$$

$$\begin{array}{lll}d(e_{3i-1})-ad_a(e_{3i-1})
&=&((i-1)\alpha_1+i\beta_{2})e_{3i-1}-\sum\limits_{k=2}^{t}\alpha_{3k-2}e_{3k+3i-4}
+\sum\limits_{k=1}^{t}\beta_{3k}e_{3k+3i-3}\\
&&\hfill-[e_{3i-1},(\alpha_1+\beta_2)x+\beta_2y+\sum\limits_{k=1}^{t}\beta_{3k}e_{3k-2}
-\sum\limits_{k=1}^{t}\alpha_{3k}e_{3k-1}+\sum\limits_{k=2}^{t}\alpha_{3k-2}e_{3k-3}]=0,\\
%
d(e_{3i})-ad_a(e_{3i})
&=&i(\beta_{2}+ \alpha_{1})e_{3i}-\sum\limits_{k=1}^{t}(\beta_{3k}e_{3k+3i-2}+\alpha_{3k}e_{3k+3i-1})\\
&&\hfill-[e_{3i},(\alpha_1+\beta_2)x+\beta_2y+\sum\limits_{k=1}^{t}\beta_{3k}e_{3k-2}
-\sum\limits_{k=1}^{t}\alpha_{3k}e_{3k-1}+\sum\limits_{k=2}^{t}\alpha_{3k-2}e_{3k-3}]=0,\\
%
d(x)-ad_a(x)&=&-\sum\limits_{k=1}^{t}k\beta_{3k}e_{3k-2}
+\sum\limits_{k=2}^{t}(k-1)\alpha_{3k}e_{3k-1}+\sum\limits_{k=2}^{t}(1-k)\alpha_{3k-2}e_{3k-3}\\
&&\hfill-[x,(\alpha_1+\beta_2)x+\beta_2y+\sum\limits_{k=1}^{t}\beta_{3k}e_{3k-2}
-\sum\limits_{k=1}^{t}\alpha_{3k}e_{3k-1}+\sum\limits_{k=2}^{t}\alpha_{3k-2}e_{3k-3}]=0,\\
%
d(y)-ad_a(y)&=&\sum\limits_{k=1}^{t}\beta_{3k}e_{3k-2}+\sum\limits_{k=1}^{t}\alpha_{3k}e_{3k-1}\\
&&\hfill-[y,(\alpha_1+\beta_2)x+\beta_2y+\sum\limits_{k=1}^{t}\beta_{3k}e_{3k-2}
-\sum\limits_{k=1}^{t}\alpha_{3k}e_{3k-1}+\sum\limits_{k=2}^{t}\alpha_{3k-2}e_{3k-3}]=0.\end{array}$$
%
\end{proof}

\begin{thm}\label{thm5}
The potential solvable Lie algebra $R_{\mathbf{n}_2}(0)$ is complete.
\end{thm}

\begin{proof}
The proof is similar to that of Theorem \ref{thm2}, but this time as the element $a$ we take
%
$$a=\alpha_1 x+(2\alpha_1+\beta_2)y-\sum\limits_{k=1}^{t}\alpha_{8k+1}f_{8k}+\sum\limits_{k=0}^{t}(-\beta_{8k+3}f_{8k+1}+\alpha_{8k+3}f_{8k+2}+\alpha_{8k+4}f_{8k+3}-$$
\hfill$-\frac{1}3\alpha_{8k+5}f_{8k+4}+\beta_{8k+7}f_{8k+5}-\beta_{8k+8}f_{8k+6}+\alpha_{8k+8}f_{8k+7}).$

%
\end{proof}

Now we prove that the second (co)homology groups of the algebras $R_{\mathbf{n}_1}(0)$ and $R_{\mathbf{n}_2}(0)$ are trivial.
\begin{thm}\label{thm3}
$H^2(R_{\mathbf{n}_1}(0),R_{\mathbf{n}_1}(0))=0.$
\end{thm}
\begin{proof} Let $\varphi\in Z^2(R_{\mathbf{n}_1}(0),R_{\mathbf{n}_1}(0))$. We show that there exists a $f\in \text{Hom}(R_{\mathbf{n}_1}(0),R_{\mathbf{n}_1}(0))$ such that $\varphi=df.$ An element of $\varphi$ of $Z^2(R_{\mathbf{n}_1}(0),R_{\mathbf{n}_1}(0))$ on the basis $\{x,y, e_1, e_2, ...\}$ is written in the form
$$\begin{array}{lll} \varphi(e_i,e_j)&=&\sum\limits_{k=1}^{p(i,j)}(\alpha_{i,j}^{3k-2}e_{3k-2}
+\alpha_{i,j}^{3k-1}e_{3k-1}+\alpha_{i,j}^{3k}e_{3k})+A_{i,j}^{1}x+A_{i,j}^{2}y,\\
\varphi(e_i,x)&=&\sum\limits_{k=1}^{s(1,i)}(\beta_{1,i}^{3k-2}e_{3k-2}+\beta_{1,i}^{3k-1}e_{3k-1}
+\beta_{1,i}^{3k}e_{3k})+B_{1,i}^{1}x+B_{1,i}^{2}y,\\
\varphi(e_i,y)&=&\sum\limits_{k=1}^{s(2,i)}(\beta_{2,i}^{3k-2}e_{3k-2}+\beta_{2,i}^{3k-1}e_{3k-1}
+\beta_{2,i}^{3k}e_{3k})+B_{2,i}^{1}x+B_{2,i}^{2}y,\end{array}$$
\hfill where $i,j \in \mathbb{N}$

and

\qquad \qquad \ \ \qquad$\varphi(x,y)=\sum\limits_{k=1}^{p}(\gamma^{3k-2}e_{2k-2}+\gamma^{3k-1}e_{3k-1}+\gamma^{3k}e_{3k})+C^{1}x+C^{2}y.$

We choose $f\in \text{Hom}(R_{\mathbf{n}_1}(0),R_{\mathbf{n}_1}(0))$ as follows

$$\begin{array}{lll}f(e_{3i-2})&=&-\sum\limits_{k=1}^{s(2,3i-2)}(\beta_{2,3i-2}^{3k-1}e_{3k-1}+\beta_{2,3i-2}^{3k}e_{3k})+a^{3i-2}e_{3i-2}+
\sum\limits_{k=1,k\neq i}^{s(1,{3i-2})}\frac{1}{i-k}\beta_{1,3i-2}^{3k-2}e_{3k-2}\\
&&\hfill -B_{2,3i-2}^1x-B_{2,3i-2}^2y,\\
f(e_{3i-1})&=&\sum\limits_{k=1}^{s(2,3i-1)}(\frac{1}{2}\beta_{2,3i-1}^{3k-2}e_{3k-2}+\beta_{2,3i-1}^{3k}e_{3k})+a^{3i-1}e_{3i-1}+
\sum\limits_{k=1,k\neq i}^{s(1,{3i-1})}\frac{1}{i-k}\beta_{1,3i-1}^{3k-1}e_{3k-1}\\
&&\hfill +B_{2,3i-1}^1x+B_{2,3i-1}^2y,\\
f(e_{3i})&=&\sum\limits_{k=1}^{s(2,3i)}(\beta_{2,3i}^{3k-2}e_{3k-2}-\beta_{2,3i}^{3k-1}e_{3k-1})+a^{3i}e_{3i}+
\sum\limits_{k=1,k\neq i}^{s(1,{3i})}\frac{1}{i-k}\beta_{1,3i}^{3k}e_{3k}+\frac{1}{i}B_{1,3i}^1x+\frac{1}{i}B_{1,3i}^2y,\end{array}$$
\hfill where $i,j \in \mathbb{N}$

\quad $\begin{array}{lll}f(x)&=&\sum\limits_{k=1}^{p}(\gamma^{3k-2}e_{3k-2}-\gamma^{3k-1}e_{3k-1})+F_{1,1}x+F_{1,2}y,\\
f(y)&=&\sum\limits_{k=1}^{p}\frac{1}{k}\gamma^{3k}e_{3k}+F_{2,1}x+F_{2,2}y,\end{array}$

with the constraints
$$a^{3i}+a^{3i-1}-a^{3(i+j)-1}=\alpha_{3i-1,3j}^{3(i+j)-1},$$
\ \ $$a^{3j}+a^{3i-2}-a^{3(i+j)-2}=-\alpha_{3i-2,3j}^{3(i+j)-2},$$
$$a^{3i-1}+a^{3j-2}-a^{3(i+j-1)}=-\alpha_{3i-1,3j-2}^{3(i+j-1)}.$$

Let us consider the cocycle $\psi=\varphi-df\in Z^2(R_{\mathbf{n}_1}(0),R_{\mathbf{n}_1}(0)).$ We show that $\psi$ is trivial. First of all by the definition for the cocycle $\psi$ one has

$\begin{array}{lll}\psi(e_{3i-2},y)&=&\sum\limits_{k=1}^{s(2,3i-2)}\beta_{2,3i-2}^{3k-2}e_{3k-2}+(iF_{2,1}-F_{2,2})e_{3i-2},\\
\psi(e_{3i-1},y)&=&\sum\limits_{k=1}^{s(2,3i-1)}\beta_{2,3i-1}^{3k-1}e_{3k-1}+((i-1)F_{2,1}+F_{2,2})e_{3i-1},\\
\psi(e_{3i},y)&=&\sum\limits_{k=1}^{s(2,3i)}\beta_{2,3i}^{3k}e_{3k}+iF_{2,1}e_{3i}+B_{2,3i}^1x+B_{2,3i}^2y,\\
\psi(e_{3i-2},x)&=&\sum\limits_{k=1}^{s(1,3i-2)}(\beta_{1,3i-2}^{3k-1}e_{3k-1}+\beta_{1,3i-2}^{3k}e_{3k})
+(iF_{1,1}-F_{1,2}+\beta_{1,3i-2}^{3i-2})e_{3i-2}+B_{1,3i-2}^{1}x+B_{1,3i-2}^{2}y,\\
\psi(e_{3i-1},x)&=&\sum\limits_{k=1}^{s(1,3i-1)}(\beta_{1,3i-1}^{3k-2}e_{3k-2}+\beta_{1,3i-1}^{3k}e_{3k})
+((i-1)F_{1,1}+F_{1,2}+\beta_{1,3i-1}^{3i-1})e_{3i-1}+B_{1,3i-1}^{1}x+B_{1,3i-1}^{2}y,\\
\psi(e_{3i},x)&=&\sum\limits_{k=1}^{s(1,3i)}(\beta_{1,3i}^{3k-2}e_{3k-2}+\beta_{1,3i}^{3k-1}e_{3k-1})
+(iF_{1,1}+\beta_{1,3i}^{3i})e_{3i},\\
\psi(x,y)&=&C^1x+C^2y. \quad \alpha_{3i-1,3j}^{3(i+j)-1}=0, \ \alpha_{3i-2,3j}^{3(i+j)-2}=0,\ \alpha_{3i-1,3j-2}^{3(i+j-1)}=0.\end{array}$

Secondly, if we impose to $\psi$ the cocycle identities $Z=0$ (see (\ref{Coc})) we derive a set of constraints.

\begin{center}
    \begin{tabular}{lll}
       \qquad\quad 2-cocyle identity & &\qquad\qquad\qquad Constraints\\
        \hline \hline
\\
        $Z(e_{3i},x,y)=0,\ i \geq1$ &\quad $\Rightarrow $\quad & $\left\{\begin{array}{ll}
                                                                 C^1=0,\ \beta_{2,3i}^{3k}=0,\ \ k\neq i, \ B_{2,3i}^1=B_{2,3i}^2=0,\\[1mm]
                                                                  \beta_{1,3i}^{3k-2}=\beta_{1,3i}^{3k-1}=0,\ \ k\geq1,\\[1mm]
                                                                 \end{array}\right.$\\
        $Z(e_{3i-2},x,y)=0,\ i\geq1 $ &\quad $\Rightarrow $\quad & $\left\{\begin{array}{ll}
                                                                 C^2=0,\  \beta_{2,3i-2}^{3k-2}=0, \ k\neq i,\  B_{1,3i-2}^1=B_{1,3i-2}^2=0,\\[1mm]
                                                                \beta_{1,3i-2}^{3k-1}=\beta_{1,3i-2}^{3k}=0,\ \ k\geq1,\\[1mm]
                                                                 \end{array}\right.$\\
        $Z(e_{3i-1},x,y)=0,\ i\geq1 $ &\quad $\Rightarrow $\quad & $\left\{\begin{array}{ll}
                                                                 \beta_{2,3i-1}^{3k-1}=0,\ \ k\neq i,\ \ B_{1,3i-1}^1=B_{1,3i-1}^2=0,\\[1mm]
                                                                \beta_{1,3i-1}^{3k-2}=\beta_{1,3i-1}^{3k}=0,\ \ k \geq 16. \\[1mm]
                                                                 \end{array}\right.$\\
             \end{tabular}
\end{center}
         \begin{center}
    \begin{tabular}{lll}
        \qquad\quad 2-cocyle identity & & \qquad\qquad\qquad Constraints \\
        \hline \hline
\\

        $Z(e_{3i-1},e_{3j},y)=0,\  i,j\geq1,$ &\quad $\Rightarrow $\quad & $\left\{\begin{array}{ll}
                                                                 \alpha_{3i-1,3j}^{3k-2}=\alpha_{3i-1,3j}^{3k}=A_{3i-1,3j}^1=A_{3i-1,3j}^2=0,\ k\geq1,\\[1mm]
                                                                 \beta_{2,3(i+j)-1}^{3(i+j)-1}=\beta_{2,3j}^{3j}+\beta_{2,3i-1}^{3i-1}, \
                                                                 \Rightarrow \beta_{2,3i-1}^{3i-1}=(i-1)\beta_{2,3}^{3}+\beta_{2,2}^{2},\\[1mm]
                                                                 \end{array}\right.$\\
        $Z(e_{3i-2},e_{3j},y)=0,\  i,j\geq1,$ &\quad $\Rightarrow $\quad & $\left\{\begin{array}{ll}
                                                                 \alpha_{3i-2,3j}^{3k-1}=\alpha_{3i-2,3j}^{3k}=A_{3i-2,3j}^1=A_{3i-2,3j}^2=0,\ k\geq1,\\[1mm]
                                                                 \beta_{2,3(i+j)-2}^{3(i+j)-2}=\beta_{2,3j}^{3j}+\beta_{2,3i-2}^{3i-2},\Rightarrow
                                                                 \beta_{2,3i-2}^{3i-2}=(i-1)\beta_{2,3}^{3}+\beta_{2,1}^{1},\\[1mm]
                                                                 \end{array}\right.$\\
        $Z(e_{3i-1},e_{3j-2},y)=0,\  i,j\geq1,$ &\quad $\Rightarrow $\quad & $\left\{\begin{array}{ll}
                                                                 \alpha_{3i-1,3j-2}^{3k-2}=\alpha_{3i-1,3j-2}^{3k-1}=0,\ k\geq1,\\[1mm]
                                                                 \beta_{2,3i}^{3i}=(i-1)\beta_{2,3}^3+\beta_{2,2}^2+\beta_{2,1}^1,\Rightarrow
                                                                 \beta_{2,3}^3=\beta_{2,2}^2+\beta_{2,1}^1.\\[1mm]
                                                                 \end{array}\right.$\\
    \end{tabular}
\end{center}

Let
$$F_{2,1}=-(\beta_{2,2}^2+\beta_{2,1}^1),\ \ F_{2,2}=\beta_{2,2}^2.$$
Then we have,
$$\psi(e_{3i-1},y)=0, \quad \psi(e_{3i-2},y)=0, \quad \psi(e_{3i},y)=0.$$

Now we consider the following identities and get constraints
\begin{center}
    \begin{tabular}{lll}
      \qquad\quad  2-cocyle identity & &\qquad\qquad\qquad Constraints\\
        \hline \hline
\\
        $Z(e_{3i-1},e_{3j},x)=0,\  i,j\geq1,$ &\quad $\Rightarrow $\quad & $ \begin{array}{ll}
                                                                 \alpha_{3i-1,3j}^{3k-1}=0,\ k\geq i+j,\  \beta_{1,3(i+j)-1}^{3(i+j)-1}=\beta_{1,3j}^{3j}+\beta_{1,3i-1}^{3i-1},\\[1mm]
                                                                 \end{array}$\\
        $Z(e_{3i-2},e_{3j},x)=0,\  i,j\geq1,$ &\quad $\Rightarrow $\quad & $ \begin{array}{ll}
                                                                 \alpha_{3i-1,3j}^{3k-2}=0,\ k\geq i+j,\  \beta_{1,3(i+j)-2}^{3(i+j)-2}=\beta_{1,3j}^{3j}+\beta_{1,3i-2}^{3i-2},\\[1mm]
                                                                 \end{array}$\\
        $Z(e_{3i-1},e_{3j-2},x)=0,\  i,j\geq1,$ &\quad $\Rightarrow $\quad & $\left\{\begin{array}{ll}
                                                                 \alpha_{3i-1,3j-2}^{3k}=0,\ \ k\neq i+j-1,\ A_{3i-1,3j-2}^1=A_{3i-1,3j-2}^2=0,\\[1mm]
                                                                 \beta_{1,3(i+j-1)}^{3(i+j-1)}=\beta_{1,3j-2}^{3j-2}+\beta_{1,3i-1}^{3i-1},\\[1mm]
                                                                 \end{array}\right.$\\
        $Z(e_{3i},e_{3j},y)=0,\  i,j\geq1,$ &\quad $\Rightarrow $\quad & $ \begin{array}{ll}
                                                                 \alpha_{3i,3j}^{3k-1}=\alpha_{3i,3j}^{3k-2}=0, \ k\geq1,\\[1mm]
                                                                 \end{array}$\\
        $Z(e_{3i-1},e_{3j-1},y)=0,\  i,j\geq1,$ &\quad $\Rightarrow $\quad & $ \begin{array}{ll}
                                                                 \alpha_{3i-1,3j-1}^{k}=0, \ k\geq1,\\[1mm]
                                                                 \end{array}$\\
        $Z(e_{3i-2},e_{3j-2},y)=0,\  i,j\geq1,$ &\quad $\Rightarrow $\quad & $ \begin{array}{ll}
                                                                \alpha_{3i-2,3j-2}^{k}=0, \ k\geq1,\\[1mm]
                                                                 \end{array}$\\
        $Z(e_{3i},e_{3j},e_{3k-1})=0,\  i,j\geq1,$ &\quad $\Rightarrow $\quad & $ \begin{array}{ll}
                                                                \alpha_{3i,3j}^{3k}=0, \ k\geq1.\\[1mm]
                                                                 \end{array}$\\
    \end{tabular}
\end{center}

Therefore we obtain
$$\beta_{1,3i-1}^{3i-1}=(i-1)\beta_{1,3}^3+\beta_{1,2}^{2}, \ \ \beta_{1,3i-2}^{3i-2}=i\beta_{1,3}^{3}-\beta_{1,2}^{2}, \ \beta_{1,3i}^{3i}=i\beta_{1,3}^{3}.$$

Finally, if we let $$F_{1,1}=-\beta_{1,3}^3,\ \ F_{1,2}=-\beta_{1,2}^2$$ then we get
$$\psi(e_{3i-1},x)=0, \quad \psi(e_{3i-2},x)=0, \quad \psi(e_{3i},x)=0.$$
Hence, $\psi\equiv 0.$
\end{proof}

\begin{thm}\label{thm4}
$H^2(R_{\mathbf{n}_2}(0),R_{\mathbf{n}_2}(0))=0.$
\end{thm}
\begin{proof} The proof of the theorem is the similar to that of Theorem \ref{thm3}.
\end{proof}

\end{document}